\documentclass[10pt]{amsart}
\usepackage{amssymb,latexsym,amscd,amsmath,amsthm,manfnt}
\usepackage{mathdots}
\usepackage[matrix,arrow]{xy}
\input{amssym.def}
\input{xy}

\xyoption{all}

\theoremstyle{definition}

\numberwithin{equation}{subsection} 
\newtheorem{guess}{theorem}[subsection]

\newtheorem{thm}[guess]{Theorem}
\newtheorem{lem}[guess]{Lemma}
\newtheorem{prop}[guess]{Proposition}
\newtheorem{Cor}[guess]{Corollary}

\newcommand{\gfr}{\mathfrak{g}}

\newcommand{\cY}{\mathcal{Y}}

\newcommand{\cZ}{\mathcal{Z}}

\newcommand{\cI}{\mathcal{I}}

\newcommand{\cQ}{\mathcal{Q}}
\newcommand{\cO}{\mathcal{O}}
\newcommand{\cC}{\mathcal{C}}
\newcommand{\cE}{\mathcal{E}}
\newcommand{\cG}{\mathcal{G}}

\newcommand{\cF}{\mathcal{F}}
\newcommand{\cH}{\mathcal{H}}
\newcommand{\cM}{\mathcal{M}}
\newcommand{\cA}{\mathcal{A}}

\newcommand{\cR}{\mathcal{R}}

\newcommand{\cX}{\mathcal{X}}
\newcommand{\cP}{\mathcal{P}}

\newcommand{\bs}{\mathbf{s}}
\newcommand{\bb}{\mathbf{b}}

\newcommand{\Aut}{\mathrm{Aut}}

\newcommand{\Spec}{\mathrm{Spec}}

\newcommand{\ob}{\mathrm{ob}}

\newcommand{\lra}{\longrightarrow}

\newcommand{\hra}{\hookrightarrow}

\newcommand{\ra}{\rightarrow}
\newcommand{\ol}{\overline}

\newcommand{\ms}{\mapsto}

\newcommand{\ul}{\underline}

\newcommand{\RR}{\mathbb{R}}

\newcommand{\ZZ}{\mathbb{Z}}
\newcommand{\GG}{\mathbb{G}}

\newcommand{\CC}{\mathbb{C}}

\newcommand{\Hom}{\mathrm{Hom}}

\newcommand{\x}{\mathrm{x}}

\begin{document}

\title{ \'Etale Fundamental group of moduli of torsors under Bruhat-Tits group scheme over a curve}

\author[A. J. Parameshwaran]{A. J. Parameshwaran}
\address{Tata Institute of Fundamental Research, Mumbai }
\email{param@math.tifr.res.in}

\author[Y. Pandey]{Yashonidhi Pandey}
\thanks{The support of Science and Engineering Research Board under Mathematical Research Impact Centric Support File number: MTR/2017/000229 is gratefully acknowledged.}
\address{ 
Indian Institute of Science Education and Research, Mohali Knowledge city, Sector 81, SAS Nagar, Manauli PO 140306, India}
\email{ ypandey@iisermohali.ac.in, yashonidhipandey@yahoo.co.uk}

\begin{abstract} Let $X$ be a smooth projective curve over an algebraically closed field $k$. Let $\cG$ be a Bruhat-Tits group scheme on $X$ which is generically semi-simple and trivial. We show that the \'etale fundamental group of the moduli stack $\cM_X(\cG)$ of  torsors under $\cG$ is isomorphic to that of the moduli stack $\cM_X(G)$ of principal $G$-bundles. For any smooth, noetherian and irreducible stack $\cX$, we show that an inclusion of an open substack $\cX^\circ$, whose complement has codimension at least two, will induce an isomorphism of \'etale fundamental group. Over $\CC$, we show that
 the open substack of regularly stable torsors in $\cM_X(\cG)$ has complement of codimension at least two when $g_X \geq 3$. As an application, we show that the moduli space $M_X(\cG)$ of $\cG$-torsors is simply-connected. 

\end{abstract}
\maketitle


\section{Introduction}
In \cite{sga1}, to a pointed scheme, Grothendieck associated a profinite group. These are called \'etale fundamental group of a scheme. His ideas are general enough to apply to any topos. In particular, they apply to the topos associated to the \'etale or smooth site of an algebraic stack. Developping on Grothendieck's ideas, Noohi in \cite{noohi} has defined the \'etale fundamental group of an algebraic stack and fibrations between them without using the language of topoi. In this paper, we follow his exposition. 

Let us now explain our setup. Let $X$ be a smooth projective curve over an algebraically closed field $k$. Pappas and Rapoport  have introduced a global Bruhat-Tits group scheme $\cG \ra  X$ \cite{prq}. Let $G$ be an almost simple simply-connected  group over $k$. Let $\cR$ denote a finite non-empty set of  points on $X$. Let $X^\circ= X \setminus \cR$. For $x \in X$, let $\mathbb{D}_x=\Spec(\hat{\cO_x})$, let $K_x$ be the quotient field of $\hat{\cO_x}$ and let $\mathbb{D}^\circ_x=\Spec(K_x)$. In this paper,  by a Bruhat-Tits group scheme $\cG \ra X$ we shall mean that  $\cG$ restricted to ${X^\circ}$ is isomorphic to $ X^\circ \times G$, and for any closed point $x \in X$,  $\cG$ restricted to $\mathbb{D}_x$ is a parahoric group scheme (cf \S \ref{gpsch}) such that the gluing functions take values in $Mor(\mathbb{D}^\circ_x,G)=G(K_x)$.  Let $\cM_X(\cG)$ denote the moduli stack of Bruhat-Tits group scheme torsors on $X$.

We construct a group scheme $\cG^{\mathbf{a}}$ which is a subgroup scheme of $\cG$ and which is Iwahori type in formal neighbourhoods of points of $\cR$. We construct morphisms $\cM_X(\cG^{\mathbf{a}}) \ra \cM_X(\cG)$ and $\cM_X(\cG^{\mathbf{a}}) \ra \cM_X(G)$ and show them to be fibrations. These are summarized in the diagram below, where the middle bride is called the Hecke correspondence by Balaji-Seshadri \cite{bs}:
\begin{equation*}
\xymatrix{
&& \cM_X(\cG^{\mathbf{a}}) \ar[rd] \ar[ld] & \cF l_{v_0} \ar[d] & \\
\cM_X^{rs}(\cG) \ar@{^{(}->}[r] \ar[d] & \cM_X(\cG) \ar[d] && \cM_X(G) \ar[d] & \cM^{rs}_X(G) \ar[d] \ar@{^{(}->}[l] \\
M_X^{rs}(\cG) \ar@{^{(}->}[r] & M_X(\cG) & &M_X(G) & M_X^{rs}(G) \ar@{^{(}->}[l]
}
\end{equation*}
Here $\cF l_{v_0}$ is the flag variety (cf (\ref{lgpflv})) and $M_X(\cG)$ and $M_X(G)$ are the moduli spaces. The square on the left only exists over the complex numbers as the notion of semi-stability is only defined so far over $\CC$ \cite{bs}. However the square on the right exists over any algebraically closed field.

Our first main result says that the \'etale fundamental group of $\cM_X(\cG)$ is isomorphic to that of $\cM_X(G)$. We also show  that $\pi_1(\cM_X(G))$ surjects onto $\pi_1(M_X(G))$ when  $g_X \geq 4$. Our main tool establishes that for a smooth stack $\cX$,  the inclusion $i: \cX^\circ \ra \cX$ of an open substack whose complement has codimension atleast two induces an isomorphism on $\pi_1$.

When $k=\CC$, as a corollary to \cite{bmp}, it follows that $\pi_1(\cM_X(\cG))$ is trivial. Then we show that $\pi_1(M_X(\cG))$  is trivial when $g(X) \geq 3$.



Let us make a remark on an alternate proof strategy. By \cite{bs}, over the complex numbers given $\cG$ there exists a Galois cover $p: Y \ra X$ with Galois group $\Gamma$ and a topological type $\tau$ such that the moduli stack of $\Gamma$-$G$ bundles of type $\tau$ is isomorphic to $\cM_X(\cG)$. This isomorphism of stacks does not seem to be adequate for computing the \'etale fundamental group. Given a stack $\cX$ together with a finite group action, there does not seem to be any relation, in general, between the \'etale fundamental group of $\cX$ and that of its fixed point substack $\cX^\Gamma$.

\subsection{Acknowledgement} The second author thanks Indranil Biswas and Arjun Paul for discussions and TIFR, Mumbai.

\section{Fundamenal group of stacks} \label{noohires}

In this preparatory section, we recall relevant notions and also fix notation.

\subsection{Algebraic Stacks} Our reference for the theory of stacks is \cite{lmb}. Let $S$ be an arbitrary scheme. In the rest of the paper it will often be $\Spec(k)$. Let $\cC$ be the site defined by endowing the category $Sch/S$ of schemes on $S$ with fppf topology. Let $\cX$ be a connected algebraic stack over $\cC$. Let $P:X \ra \cX$ be a smooth atlas of $\cX$. To $P$ we can associate a {\it groupoid-space} $X_\bullet$ namely $X \times_{\cX} X \xymatrix{ {} \ar@<1ex>[r] \ar[r] & {} } X$ (cf \cite[(2.4.3)]{lmb}). To this we can associate the $S$-groupoid $[X_\bullet]'$. It is a prestack whose associated stack $[X_\bullet]$ is $\cX$ itself (\cite[(3.4.3)]{lmb}). We will often reduce questions on stacks $\cX$ to grouppoid-spaces $X_\bullet$. Let us see reinterpretations of morphisms and $2$-morphisms of stacks in terms of groupoid-spaces. These will be used to prove that restriction from an algebraic stack $\cX$ to an open substack $\cX^\circ$ whose complement has codimension at least two induces an isomorphism on the \'etale fundamental groups. 

\subsubsection{Morphism of stacks} \label{morphism} Let $a: \cY \ra \cZ$ be a  morphism of algebraic stacks. It is said to be representable if for every morphism $U \ra \cZ$ where $U \in \ob(\cC)$, the fiber product $U \times_{\cZ} \cY$ is representable by an algebraic space. Let $p_Y: Y \ra \cY$ and $p_Z: Z \ra \cZ$ be  smooth atlases of $\cY$ and $\cZ$ respectively. By taking their presentations if necessary, we may suppose that they are schemes. Consider $a \circ p_Y: Y \ra \cY \ra \cZ$. The morphism $p_Z$ being smooth admits \'etale local sections. Replacing $Y$ by its  \'etale refinements if necessary, we may suppose that $a \circ p_Y$ equals $p_Z \circ A$ for some $A: Y \ra Z$. For the morphism $a$ we call such a choice of atlases $p_Y: Y \ra \cY$ and $p_Z: Z \ra \cZ$ as {\it adapted}. In terms of such adapted atlases, the morphism $a$ equivalently corresponds to equivalence classes of morphisms $A: Y \ra Z$ of schemes satisfying the condition $p_Z \circ A \circ q_0=p_Z \circ A \circ q_1$ i.e 
\begin{equation} \label{1.4Knutson}
 \xymatrix{ Y \times_{\cY} Y \ar@<1ex>[r]^{q_0} \ar[r]_{q_1} & Y \ar[r]^A & Z \ar[r]^{p_Z} & \cZ}
\end{equation} 
The equivalence relation on $A$ is the following: Let $A_0,A_1: Y \ra Z$ be morphisms satisfying (\ref{1.4Knutson}). Their compositions with $Z \ra \cZ$ agree if and only if there exists a $h$ as above such that $A_0=p_0 \circ h$ and $A_1= p_1 \circ h$. 
\begin{equation} \label{mordiag}
\xymatrix{
 & Y \ar@{.>}[ld]_{h}  \ar[r]^{p_Y} \ar[d]^{A_0} \ar@<1ex>[d]_{A_1} & \cY \ar[d]^a \\
Z \times_{\cZ} Z \ar@<1ex>[r]^{p_0} \ar[r]_{p_1} & Z \ar[r]^{p_Z} & \cZ
}
\end{equation}

Notice that (\ref{1.4Knutson}) is equivalent to the existence of a morphism 
\begin{equation} \label{realKnutson} A':Y \times_{\cY} Y \ra Z \times_{\cZ} Z
\end{equation} such that $p_i \circ A'= A \circ q_i$ for $i \in \{0,1\}$.  Whenever (\ref{realKnutson}) holds we get moreover a morphism $A_\bullet: cosk_{\cY}(Y) \ra cosk_{\cZ}(Z)$ where $cosk_{\cY}(Y)$ denotes the coskeleton functor. 

\subsubsection{$2$-morphisms} \label{twomorphism}
Let $f$ and $g$ be two maps from $\cY \ra \cZ$. We chose atlases of $\cY$ and $\cZ$ which are adapted for both of them. Let $F$ and $G$ represent $f$ and $g$ in the sense of \S \ref{morphism}. Let $F'$ and $G'$ be the induced maps $Y \times_{\cY} Y \ra Z \times_{\cZ} Z$.
 
By a simple homotopy $H_\Phi$ from $F$ to $G$ we mean two morphisms $H^0_\Phi,H^1_\Phi$ 
\begin{equation} \label{homotopy} H_{\Phi} :=\{ H_\Phi^0, H_\Phi^1 \}: Y  \ra X \times_{\cX} X 
\end{equation}  as below where the morphisms $p_0,p_1,q_0,q_1$ are understood
\begin{equation}
\xymatrix{
Y \times_{\cY} Y  \ar@<1ex>[d]^{F'} \ar[d]_{G'} \ar@<1ex>[r] \ar[r] & Y \ar@<1ex.>[ld]^{H_\Phi^0} \ar@{.>}[ld]_{H_\Phi^1}  \ar[r] \ar@<1ex>[d]^{F } \ar[d]_{G} & \cY \ar@<1ex>[d]^{f} \ar[d]_{g} \\
X \times_{\cX} X \ar@<1ex>[r] \ar[r] & X \ar[r] & \cX
}
\end{equation}
such that we have the following equality of maps $Y \times_{\cY} Y \ra X \times_{\cX} X$:
\begin{eqnarray} H^0_\Phi \circ q_0 & = & F' \\
 H^{0}_\Phi \circ q_1 & = & H^{1}_\Phi \circ q_0  \\
 H^{1}_\Phi \circ q_1 & =& G' 
\end{eqnarray}
and the following equality of maps $Y \ra X$
\begin{eqnarray}
 p_0 \circ H_\Phi^0 & =& F \\ 
 p_1 \circ H_\Phi^0 & = & p_0 \circ H_\Phi^{1}  \\ 
 p_1 \circ H_\Phi^1 & = & G.
 \end{eqnarray} 
This is just the restriction of a simplicial homotopy from $F_\bullet$ to $ G_\bullet$ viewed as morphisms between the simplicial schemes $cosk_{\cY}(Y)$ and $cosk_{\cZ}(Z)$.
It is also clear when and how one simple homotopy is composable with another.

Consider a $2$-morphism $\Phi: f \ra g$ between morphisms $f,g: \cY \ra \cX$. 
Such a $2$-morphism corresponds to a sequence $H_\Phi= \{H^0_\Phi,\cdots, H^k_\Phi \}$ of composable homotopies from $F$ to $G$. Here $k \geq 0$ must be an odd integer. We shall call $H_\Phi$ a {\it homotopy} from $F$ to $G$ (or $f$ to $g$) corresponding to $\Phi$. Conversely, any such homotopy $H$ gives a $2$-morphism which we will denote as $\Phi_H$.

\subsubsection{Covering map} \label{covering}  A representable morphism $f: \cY \ra \cX$ is said to be finite (resp.  \'etale) if for every $U \ra \cX$ as above, the morphism $U \times_{\cX} \cY \ra U$ is finite (resp. \'etale). A morphism $f:\cY \ra \cX$ of connected algebraic stacks is called a covering map if $f$ is representable, finite and \'etale map. 

\subsubsection{Dimension} \label{dim} We refer the reader to \cite[page 98,99]{lmb}. The notion of dimension of an algebraic stack $P: X \ra \cX$ is defined when $\cX$ is locally noetherian. Similarly for a point $x \in \cX$, relative dimension $dim_x(P)$ of $P$ at $x$ is defined. Now $dim_{x} \cX=dim X - dim_x P$. The function $x \ms dim_x(P)$ is compatible with base-change and is continous.

\subsubsection{Local construction} \label{lc} For $U \in ob(\cC)$, consider the site $Sch/U$ with fppf topology. By a $U$-space we mean a sheaf of sets on the site $Sch/U$. Let $\underline{Es}_U$ denote the category of $U$-spaces. Let $\underline{Es} \ra Sch/S$ denote the fibered category which for every $U \in ob(\cC)$ associates the category $\ul{Es}_U$.  Recall \cite[Defn 14.1.2]{lmb} that a local construction on a $S$-stack $\cX$ is a functor 
\begin{equation} \label{cartfunc}
\underline{\cY}: \cX \ra \ul{Es}
\end{equation}
satisfying the following condtion: for every $U \in ob(Sch/S)$ we have a functor
 \begin{equation}
 \ul{\cY}_U: \cX_U \ra \ul{Es}_U
 \end{equation}
 and for every $\phi: V \ra U$ we have a natural transformation $\epsilon_\phi$
 \begin{equation} \label{2morcart}
 \xymatrix{
 \cX_U \ar[rr]^{\phi^*} \ar[d]^{\ul{\cY}_U} && \cX_V \ar[d]^{\ul{\cY}_V} \ar@{=>}[lld]^{\epsilon_\phi} \\
 \ul{Es}_U \ar[rr]_{V \times_{\phi,U}(-)} && \ul{Es}_V
 }
 \end{equation}
such that for any $\psi: W \ra V$ we have
\begin{equation} \label{coc}
\epsilon_{\phi \circ \psi} = (W \times_{\psi,V} \epsilon_\phi) \circ \epsilon_\psi(\phi^*).
\end{equation}
We recall that to a local construction $\ul{\cY}$ we can associate a $S$-groupoid $\cY$ with a natural map to $\cX$ (cf \cite[(14.1.4)]{lmb}). By  \cite[Prop 14.1.5]{lmb} $\cY$ is a $S$-stack. A local construction $\ul{\cY}$ is said to be {\it algebraic} if for every $U \in ob(Sch/S)$, and every $x \in ob(\cX_U)$, the $U$-space $\cY_U(x)$ is algebraic. For a local algebraic construction $\ul{\cY}$ the morphism of $S$-stacks $\cY \ra \cX$ is representable. Further the morphism $\cY \ra \cX$ possesses a property $P$ if for every $U \in ob(Sch/S)$ and every $x \in ob(\cX_U)$, the morphism  $\cY_U(x) \ra U$ of algebraic spaces possesses $P$ (cf \cite[Prop (14.1.7)]{lmb}).

\subsection{Galois categories} \label{gc} Let us recall the Galois category $C_\cX$ associated to $\cX$. Its objects consist of pairs $(\cY,f)$ where $ f: \cY \ra \cX$ is a covering map. A morphism from $(\cY,f)$ to $(\cZ,g)$ consists of equivalence classes of pairs $(a,\Phi)$ where $a: \cY \ra \cZ$ is a morphism of stacks $\Phi: f \ra g \circ a$ is a $2$-morphism. Here we say that $(a, \Phi) \sim (b,\Psi)$ if there exists a $2$-morphism $\Gamma: a \ra b$ such that $g(\Gamma) \circ \Phi=\Psi$.

\subsection{Fundamenal functor} \label{ff} Let $x,x': \Spec(k) \ra \cX$ be two  geometric points of $\cX$. A hidden path from $x$ to $x'$ is defined to be a transformation from $x$ to $x'$. Let us recall the fundamental functor $F_x: C_\cX \ra \{ Sets \}$. To a covering $f: \cY \ra \cX \in Ob(C_\cX)$, it associates the set of equivalence classes of pairs $(y,\phi)$ where $y: \Spec(k) \ra \cY$ is a geometric point and $\phi: x \ra f(y)$ is a hidden path. Here we say that $(y,\phi) \sim (y',\phi')$ if there exists a $\beta: y \ra y'$ such that $f(\beta) \circ \phi=\phi'$.

Noohi shows that the pair $(C_\cX,F_x)$ forms a Galois category \cite[Thm 4.2]{noohi}.
The fundamental group $\pi_1(\cX,x)$ is defined to be the group of self-transformations of the functor $F_x$. It is a profinte group. The hidden fundamental group $\pi^h_1(\cX,x)$ is the group of self-transformation of $x: \Spec(k) \ra \cX$.

\subsection{Fibrations} \label{fibration} In this subsection, we want to define fibrations for algebraic stacks. We first start with the definition of fibrations for algebraic spaces.
Let $f: Y \ra X$ be a morphism of algebraic spaces. 

Let us recall Axiom A from \cite{noohi}: If we restrict $f$ to any connected component of $Y$ and the corresponding connected component of $X$, then for any choice of base points $y$ and $x=f(y)$, the sequence
\begin{equation} 
\pi_1(Y_x,y) \ra \pi_1(Y,y) \ra \pi_1(X,x) \ra \pi_0(Y_x,y) \ra \{*\},
\end{equation}
is exact. A surjective morphism $f: Y \ra X$ of algebraic spaces is called a fibration if every base-extension of $f$ satisfies Axiom A. 

For example, finite $\acute{e}$tale morphisms are fibrations. Also when $X$ is locally noetherian, then any proper flat morphism $f: Y \ra X$ that has geometrically connected fibers is a fibration by \cite[Expos\'e X,1.6]{sga1}.

We now define fibrations for algebraic stacks. A pointed morphism between pointed stacks $(\cY,y)$ and $(\cX,x)$ consists of a morphism $f: \cY \ra \cX$ together with a transformation $\phi: x \ra f(y)$. One says that a morphism $f: \cY \ra \cX$ of algebraic stacks is a fibration if for any choice of base points $y: \Spec(k) \ra \cY$ and $x: \Spec(k) \ra \cX$ such that $f(y)$ is $2$-isomorphic to $x$, the following sequence is exact:
\begin{equation} \label{fibseq}
\pi_1(\cY_x,y) \ra \pi_1(\cY,y) \ra \pi_1(\cX,x) \ra \pi_0(\cY_x,y) \ra \{*\}
\end{equation} 
Here by $\pi_0$ of an algebraic stack, we mean the $\pi_0$ of its atlas.

When $f: \cY \ra \cX$ is a representable morphism, then it is a fibration if and only if for any algebraic space or scheme $X \ra \cX$, the base-extension of $f$ is a fibration of algebraic spaces as defined above \cite[A.4]{noohi}. Further, by \cite[\S 8, page 88, (3)]{noohi} being a fibration is local (on the target) in the fppf topology. 


\section{Behaviour of $\pi_1$ with respect to restriction to open substacks}

Let us mention a convention. Given a morphism of algebraic stacks $a: \cX \ra \cY$ as in \S \ref{morphism}, we will choose atlases adapted for the morphism $a$ and denote the corresponding morphism between them by capital letters $A:X \ra Y$ as in (\ref{1.4Knutson}).

\begin{prop} \label{codim2} Let $\cX/k$ be a connected, irreducible, noetherian and smooth algebraic stack and $\cX^\circ$ be a non-empty open substack. Let $\overline{x}: \Spec(k) \ra \cX^\circ$ be a closed point. The restriction morphism $\pi_1(i): \pi_1(\cX^\circ, \overline{x}) \ra \pi_1(\cX,\overline{x})$ induced by the inclusion is surjective. When the complement of $\cX^\circ$ has codimension at least two, then $\pi_1(i)$ is  an isomorphism.
\end{prop}

\begin{proof} Since $\cX$ is irreducible, so for any irreducible covering $\tilde{\cX} \ra \cX$, the fiber product $ \cX^\circ \times_{\cX} \tilde{\cX}$ remains irreducible because it is a non-empty open substack of $\tilde{\cX}$. Since $\pi_1(\cX,\overline{x})$ is a profinite group, this show that $\pi_1(i)$ is surjective.

 Let us denote by $P: X \ra \cX$ an atlas of $\cX$. So $P$ is a representable, surjective and smooth morphism. By taking a presentation of the algebraic space $X$, we may suppose that it is a scheme. Since $\cX/k$ is smooth and noetherian, so $X$ is smooth and noetherian. We may further assume that it is a scheme of finite type over $k$. Let $P^\circ: X^\circ= \cX^\circ \times_{\cX} X \ra \cX^\circ$ be the corresponding atlas of $\cX^\circ$.

Let us first show that the Galois categories $C_\cX$ and $C_{\cX^\circ}$ have the same objects. Given an \'etale cover $\cY \ra \cX$, we get an \'etale cover of $\cX^\circ$ by taking fiber-product. Conversely, let us suppose that we have an \'etale cover $\cY^\circ \ra \cX^\circ$. Let us construct the corresponding \'etale cover of $\cX$. Set $Y^\circ = \cY^\circ \times_{\cX^\circ} X^\circ$. Since it is  an \'etale cover of $X^\circ$, it is a scheme. Let us denote by $p_1,p_2: X^\circ \times_{\cX^\circ} X^\circ \ra X^\circ$ the two projection maps to $X^\circ$.  A morphism $\cY^\circ \ra \cX^\circ$ furnishes a descent-data
\begin{equation} \label{descentdata}
\phi^\circ: p_1^*(Y^\circ) \ra p_2^*(Y^\circ),
\end{equation}
satisfying the usual co-cycle condition $p_{13}^*(\phi^\circ)=p_{23}^*(\phi^\circ) \circ p_{12}^*(\phi^\circ)$ where $p_{ij}: X^\circ \times_{\cX} X^\circ \times_{\cX} X^\circ \ra X^\circ \times_{\cX} X^\circ$ denotes the usual projection morphism.

Consider the inclusion of schemes $i: X^\circ \ra X$.  The fibers of $P$ and $P^\circ$ are isomorphic whenever those of $P^\circ$ are non-empty. Since the morphism $X \ra \cX$ is smooth, and $dim(P)$ is a continous function (cf \S \ref{dim}) so the complement  of $X^\circ \ra X$ has codimension at least two. Let us lift the given closed point $\ol{x} : \Spec(k) \ra \cX^{\circ}$ to a closed point $\ol{x}': \Spec(k) \ra X^\circ$. Therefore, $\pi(i): \pi_1(X^\circ, \overline{x}) \ra \pi_1(X, \overline{x})$ is an isomorphism of profinite groups.  Therefore there exists an \'etale  cover $Y \ra X$ such that $Y^\circ = Y \times_X X^\circ$. Similarly consider $j: X^\circ \times_{\cX^\circ} X^\circ \ra X \times_{\cX} X$. Again $j$ is an open immersion whose complement is of codimension at least two. Therefore, the morphism $\phi^\circ$ extends uniquely to a morphism
\begin{equation} \label{dd}
\phi: p_1^*(Y) \ra p_2^*(Y).
\end{equation}
Now $p_{13}^*(\phi)$ and $p_{23}^*(\phi) \circ p_{12}^*(\phi)$ agree on the open subset $X^\circ \times_{\cX^\circ} X^\circ\times_{\cX^\circ} X^\circ$ of $X \times_{\cX} X \times_{\cX} X$ because $\phi$ restricts to $\phi^\circ$. The latter is integral because $\cX/k$ is smooth. Therefore the open set where these morphisms agree is also dense. So these morphisms agree everywhere. So $\phi$ is a descent data of $Y \ra X$ over $P: X \ra \cX$. Using this we will construct now an algebraic stack $\cY$ \'etale over $\cX$.

Let us construct a functor $\ul{\cY}: \cX \ra \ul{Es}$ as in (cf (\ref{cartfunc})). For any $u: U \ra \cX$ over $k$, let $X_u$ denote $ U \times_{u,\cX,P} X$, let $Y_u$ denote $ U \times_{u,\cX,P} Y$ and let $\phi_u$ denote the corresponding pull-back from (\ref{dd}) via $u$.  The morphism $X_u \ra U$ is faithfully flat because $P$ is smooth. Also since $P$ is smooth so it is a morphism of finite type. So the morphism $X_u \ra U$ is quasi-compact. Further $Y_u \ra X_u$ is quasi-affine because it is a finite morphism obtained via base-change from the \'etale covering $\cY \ra \cX$ (cf \S \ref{covering}). Therefore by fpqc descent \cite[Chap 6, Thm 6]{blr}, we get a scheme $\ul{\cY}_u$ on $U$. Further $\ul{\cY}_u \ra U$ is also \'etale. This defines a functor $\ul{\cY}_U: \cX_U \ra \ul{Es}_U$. This functor is cartesian. Indeed, (\ref{2morcart}) says simply that descent commutes with base-change and (\ref{coc}) follows from the fact that this commutation is natural.  Therefore $\ul{\cY}$ is a local construction. It is an {\it algebraic local construction} because $\ul{\cY}_u$ is a scheme and hence an algebraic space. Let $\cY$ be the associated stack which has a natural morphism to $\cX$ (cf \S \ref{lc}). By construction, the fiber product of $\cY \times_{\cX} X$ is naturally $Y$, and so $Y \ra \cY$ is an atlas. So $\cY$ is an algebraic stack. Further the map $\cY \ra \cX$ is \'etale because $Y_u \ra U$ is \'etale (cf \cite[Defn 14.1.6]{lmb}). Therefore $\cY \ra \cX$ is an \'etale cover of stacks.

 Let us show that morphisms in the category $C_{\cX^\circ}$ and $C_{\cX}$ are the same. Morphisms in $C_{\cX}$ restrict to morphisms in $C_{\cX^\circ}$. Conversely let $(a^\circ, \Phi^\circ): (\cY^\circ, f) \ra (\cZ^\circ,g)$ represent a morphism between \'etale covers of $\cX^\circ$. Let $Y^\circ$ and $Z^\circ$ denote the fiber-products of $\cY $ and $\cZ$ with $X^\circ$. By replacing the atlases $Y^\circ$ and $Z^\circ$ by atlases adapted for the morphism  $a^\circ$ as in \S \ref{morphism}, we get a morphism $A^\circ: Y^\circ \ra Z^\circ$ (cf (\ref{morphism})). It is compatible with the descent datas $\phi^\circ_Y$ and $\phi^\circ_Z$ i.e the following diagram 
\begin{equation}
\xymatrix{
p_1^*(Y^\circ) \ar[r]^{\phi^\circ_Y} \ar[d]_{p_1^*(A^\circ)} & p_2^*(Y^\circ) \ar[d]^{p_2^*(A^\circ)} \\
p_1^*(Z^\circ) \ar[r]^{\phi^\circ_Z} & p_2^*(Z^\circ)
}
\end{equation}  
commutes. If $Y$ and $Z$ are \'etale covers of $X$ extending $Y^\circ$ and $Z^\circ$, then $A^\circ$ extends to $A:Y \ra Z$. The morphisms $ p_2^*(A) \circ \phi_Y$ and $\phi_Z \circ p_1^*(A) : p_1^*(Y) \ra p_2^*(Z)$ are equal because they agree on the open subset $p_1^*(Y^\circ)$ of $p_1^*(Y)$ which is an integral scheme. Therefore $A$ descends to a morphism $a:  \cY \ra\cZ$ of algebraic stacks. Thus $a^\circ$ extends to $a$. Thus $1$-morphisms extend.

Suppose we are given $\Phi^\circ: f^\circ \ra g^\circ \circ a^\circ$ which is a $2$-morphism. Now $f^\circ$ and $g^\circ$ extend. We seek to find a $2$-morphism $\Phi: f \ra g \circ a$ between morphisms from $\cY$ to $\cX$ extending $\Phi^\circ$. Replacing $Y$ and $Z$ by atlases adapted for the morphisms $f$ and $g$ as in \S \ref{morphism}, we get morphisms $F: Y \ra X$ and $G: Z \ra X$ representing $f: \cY \ra \cX$ and $g: \cZ \ra \cX$ in sense of diagram (\ref{mordiag}).  By \S \ref{twomorphism}, for some odd integer $k \geq 0$, the $2$-morphism $\Phi^\circ$ gives rise to a sequence of composable homotopies 
\begin{equation}
H_{\Phi^\circ}:= \{H^0_{\Phi^\circ}, \cdots, H^k_{\Phi^\circ} \}: Y \times_X X^\circ \ra (X \times_{\cX} X)_{X^\circ}
\end{equation} This extends to a homotopy $H_\Phi:=\{H_\Phi^0, \cdots H_\Phi^k\}:  Y \ra X \times_{\cX} X$ as follows. 
Firstly $p_1 H^0_{\Phi^\circ}=F^\circ$ which, we have seen before, extends as $F$. Similarly $p_2 H^0_{\Phi^\circ}$ extends. The extensions of $p_1 \circ H^0_{\Phi^\circ}$ and $p_2 \circ H^0_{\Phi^\circ}$ define a morphism $Y \ra Y \times_{\cX} Y$. They  agree on $Y^\circ$, so they agree on $Y$. Thus together they give a morphism which we denote as $H^0_{\Phi}$ which extends $H^0_{\Phi^\circ}$. Similarly for $1 \leq i \leq k$, each of $H^i_{\Phi^\circ}$ extends to $H^i_{\Phi}$. Now $p_2 \circ H^0_{\Phi^\circ}=p_1 \circ H^1_{\Phi^\circ}$ are morphism $Y^\circ \ra X^\circ$. Each of these extend individually to morphisms $Y \ra  X$.  These extensions agree on $Y^\circ$, and $Y$ is an integral scheme, so these extensions agree everywhere over $Y$. By repeating the above arguements, we can show that the extensions $H^i_{\Phi}$ satisfy the required relations of \S \ref{twomorphism}. So simple homotopies extend as simple homotopies and remain composable. 


 Let us check that extension of homotopies passes through the equivalence relation on them i.e if $(a^\circ,\Phi^\circ) \sim (b^\circ,\Psi^{\circ})$ then their extensions are also related.  If $(a^\circ,\Phi^\circ) \sim (b^\circ,\Psi^\circ)$, then by definition (cf \S \ref{gc}) there exists a $2$-morphism $\Gamma^\circ: a^\circ \ra b^\circ$ such that $g(\Gamma^\circ) \circ \Phi^\circ = \Psi^\circ$. Corresponding to $\Gamma^\circ$, for some odd integer $k \geq 0$, we have a homotopy 
\begin{equation}
H_{\Gamma^\circ}:=\{H^0_{\Gamma^\circ}, \cdots H^k_{\Gamma^\circ} \}: Y^\circ \ra (Z \times_{\cZ} Z)_{\cZ^\circ}
\end{equation}
 from $A^\circ$ to $B^\circ$. As before, it extends to a homotopy $H:=\{H_\Gamma^0, \cdots, H_\Gamma^k \}$ from $A$ to $B$. This gives a $2$-morphism $\Gamma: a \ra b$. We have to check  that $g(\Gamma) \circ \Phi=\Psi$. With notations as before consider 
\begin{equation}
\xymatrix{
Y \times_{\cY} Y \ar@<1ex>[rr] \ar[rr] && Y   \ar[rr] \ar@<1ex>[d]_A \ar[d]^B \ar[rdd]^F 
\ar@<1ex>[lld]_{H^0_\Gamma} \ar[lld]^{H^k_\Gamma}
&& \cY \ar@<1ex>[d]^{b} \ar[d]_{a} \ar[rdd]^f \\ 
Z \times_{\cZ} Z \ar@<1ex>[rr] \ar[rr] \ar[rd]^{G'} && Z  \ar[rr] \ar[rd]_G && \cZ \ar[dr]^g \\ 
& X \times_{\cX} X \ar@<1ex>[rr] \ar[rr] && X \ar[rr] && \cX } 
\end{equation}
Just as in (\ref{realKnutson}), $G: Z \ra X$ furnishes $G':Z \times_{\cZ} Z \ra X \times_{\cX} X$. By $G(H_\Gamma)$ let us agree to mean $\{G'(H^0_\Gamma), \cdots G'(H^k_\Gamma) \}$.
Just as in (\ref{homotopy}), for some odd  $j \geq 0$ let \begin{equation}  H_{\Psi} :=\{ H_\Psi^0, \cdots, H_\Psi^j \}: Y  \ra (X \times_{\cX} X) 
\end{equation}  denote a homotopy between $F$ and $G \circ B$.
It follows that $G(H_\Gamma) \circ H_{\Phi}=H_{\Psi}$. Thus we have the desired $g(\Gamma) \circ \Phi=\Psi$ extending $g(\Gamma^\circ) \circ \Phi^\circ =\Psi^\circ$.

 Let $\overline{x}: \overline{k} \ra \cX^\circ$ be the given geometric point. We want to check that the fiber functors for $\cX^\circ$ and $\cX$ are also isomorphic. This follows by repeating the preceeding arguments of construction of \'etale covers of $\cX^\circ$ and $\cX$ and applying the extension of $2$-morphisms above to hidden paths for $\cX^\circ$ and $\cX$. This shows that $\pi_1(\cX^\circ)$ is isomorphic to $\pi_1(\cX)$.
 
\end{proof}

\section{Local group theoretical data of parahoric group schemes} \label{lgpthedata}

 Let $A:=k[[t]]$ and $K:=k((t))=k[[t]][t^{-1}]$, where $t$ denotes a uniformizing parameter.
Let $G$ be a {\em semisimple simply connected affine algebraic
group} defined over $k$. We now want to consider the group $G(K)$. 

All the notions in this section hold, not only for $K$, but for a connected reductive group over a local field. This is called the twisted case in \cite{pradv}. But for simplicity, {\it we will specialize these to the untwisted case namely that of $G_k(K)$ over $K$}.  

We shall fix a maximal torus $T \,\subset\, G$ and let
$Y(T)\,=\, \Hom(\GG_m,\, T)$ denote the group of all
one--parameter subgroups of $T$. For each maximal torus $T$ of $G$, the {\it
standard affine apartment} $\cA_T$  is an affine space under $Y(T)
\otimes_\ZZ \RR$.  In \cite{bt1} there is no preferred choice of origin in $\cA_T$. But
for recalling parahoric groups schemes, {\it we may identify $\cA_T$ with $Y(T) \otimes_\ZZ \RR$}
(see \cite[\S~2]{bs}) by choosing a point $v_0 \in \cA_T$. This $v_0$ is also called an origin. For a root $r$ of $G$ and an integer $n \in \ZZ$, we get an affine functional
\begin{equation} \label{afffunc} \alpha=r + n : \cA_T \ra \RR, x \ms r(x -v_0) +n.
\end{equation}
These are called the {\it affine roots} of $G$. For any point $x \in \cA_T$, let $Y_x$ denote the set of affine roots vanishing on $x$. For an integer $n\geq 0$, define
\begin{equation} \label{facetdefn}
\cH_n=\{x \in \cA_T | |Y_x|=n \}.
\end{equation}
A facet $\sigma$ of $\cA_T$ is defined to be a connected component of $\cH_n$ for some $n$. The dimension of a facet is its dimension as a real manifold. 

Let $R\,=\,R(T,G)$ denote the root system of $G$ (cf. \cite[p. 125]{springer}). Thus for
every $r \,\in\, R$, we have
the root homomorphism $u_r \,: \, \GG_a\,\lra\, G$ \cite[Proposition 8.1.1]{springer}. For any non-empty subset
$\Theta\,\subset\, \cA_T$, the {\it parahoric subgroup} 
$\cP_\Theta\,\subset\, G(K)$ is
defined (\cite[Page 8]{bs}) as 
\begin{equation} \label{defnpara}
\cP_\Theta = <T(A) \, , \ u_r(t^{m_r}A) \,\, | \,\, m_r\,=\,m_r(\Theta) \,=\, - \lfloor {\rm inf}_{\theta \in \Theta} (\theta,r) \rfloor >.
\end{equation}  
{\it In this paper $\Theta$ will always be either a {\em facet} or a point of $\cA_T$}.

Moreover, by \cite[Section 1.7]{bt} we have an affine flat smooth group scheme $\cG_\Theta\,\lra\, \Spec(A)$ corresponding to $\Theta$. This is called the {\it parahoric group scheme} associated to $\Theta$. It satisfies the following properties that characterize it. The set of $K$--valued (respectively,
$A$--valued) points of $\cG_\Theta$ is identified with $G(K)$ (respectively,
$\cP_\Theta$). The group scheme $\cG_\Theta$ is uniquely determined by its
$A$--valued points. For a facet $\sigma \subset \cA_T$, let  $\cG_\sigma$ be the parahoric group scheme by $\sigma$.

For a facet $\sigma$, we may check (\ref{defnpara}) for any $\theta$ in  $\sigma$.  The pro-unipotent radical $\cP_\Theta^u \subset \cP_\Theta$ is defined as 
\begin{equation} \label{defnparau}
\cP_\Theta^u = <T(1+t) \, ,\ u_r(t^{1-\lceil (\theta,r) \rceil }A) | \theta \in \Theta>.
\end{equation} 
Let $\cG_\Theta^u \ra \Spec(A)$ denote the affine flat group scheme corresponding to $\cP^u_\Theta$.

\subsection{Alcove} \label{alcove}
We choose a Borel $B$ in $G/k$ containing $T$. This determines a choice of positive roots.  Let 
$\mathbf{a}_0$ denote the unique alcove in $\cA_T$ whose closure  contains $v_0$ and is contained in the finite Weyl chamber determined by positive simple roots. The affine walls defining $\mathbf{a}$ determine a set $\mathbf{S}$ of simple {\it affine roots}. We will denote these simple roots by the symbols $\{\alpha_i\}$.

\subsection{The  closed fiber of $\cG_{\sigma}$} The results of this subsection are surely known to experts. But we couldn't find a suitable reference. 
 
For a facet $\sigma$ (or a point $p \in \cA_T$) let us denote the set of affine roots vanishing at $\sigma$ by
$Y_\sigma$ (or $Y_p$). It is possible to identify $Y_\sigma$ with a closed sub root system of the root system of $G$ as in \cite{borels} and we will do so. We refer the reader to \cite[4.6.12]{bt} for


\begin{prop}  \label{raghu} The root system of the reductive quotient $\cG_\sigma/\cG^u_\sigma$ of the special fiber of $\cG_\sigma$ is given by $Y_\sigma$.
\end{prop}

Let us introduce some notations for a facet $\sigma$ of $\mathbf{a}$ and a root $r$ of $G$. By the definition of a facet (\ref{facetdefn}), for any two points $p$ and $p'$ in $\sigma$, it follows that 
\begin{equation} Y_p = Y_{p'}.
\end{equation} Therefore for a root $r \in R$, if $(p,r) \notin \ZZ$ for some $p \in \sigma$, then it holds for all other  points $p' \in \sigma$ also and $|(p,r)-(p',r)| < 1$. Since the facets are convex subsets of the Euclidean space, so by the intermediate value  theorem, this observation allows us to write expressions of the form $(\sigma,r) \notin \ZZ$, $\lfloor (\sigma,r) \rfloor$ and $\lceil (\sigma,r) \rceil$ by choosing an arbitrary point $p \in \sigma$.  For an integer $n$, a facet $\sigma$  and an affine root $\alpha$, we will write expressions of the form $n \leq \alpha(\sigma) $ (or $\alpha(\sigma) <n$) when $n \leq \alpha(p)$  ( or $\alpha(p) <n$) for all  points $p \in \sigma$.

\begin{Cor} \label{borelpar} Let $\bs$ and $\bb$ be  facets of  $\mathbf{a}$. Suppose that $\bs$ is in the closure of $\bb$. Let $G_{\bs,\bb}  \subset Y_{\bs}$ be defined by
\begin{equation} \label{Gsb}
G_{\bs,\bb}= \{ \alpha \in Y_{\bs} | 0 \leq \alpha(\bb) \}.
\end{equation}
The group $\cG_{\bb}/\cG^u_{\bs}$ is the parabolic  subgroup of $\cG_{\bs}/\cG^u_{\bs}$ given by $G_{\bs,\bb}$.
\end{Cor}
\begin{proof} Let $R$ be an $A$-algebra, where $A=k[[t]]$. Here below it will be convenient to write $A \otimes R$ (short form for $A \otimes_A R$) instead of $R$.  By (\ref{defnpara}) and (\ref{defnparau}), we have $$\cP_{\mathbf{s}}^u \subset \cP_{\mathbf{b}}^u \subset \cP_{\mathbf{b}} \subset \cP_{\mathbf{s}}.$$ Let us consider $\cP_{\mathbf{s}}^u \subset \cP_{\mathbf{b}}$. Further, we have $u_r(t^{- \lfloor (\bb,r) \rfloor} A \otimes R) \subset \cG_{\bb}(A \otimes R)$ and $u_r(t^{1- \lceil (\bs,r) \rceil} A \otimes R) \subset \cG^u_{\bs}(A \otimes R)$. So $1- \lceil (\bs,r) \rceil \geq - \lfloor (\bb,r) \rfloor$. Thus the quotient $\cG_{\bb} (A \otimes R)/\cG^u_{\bs}(A \otimes R)$ is generated by roots $r$ such that 
\begin{equation} \label{floorceiling}
\lfloor (\bb,r) \rfloor = \lceil (\bs,r) \rceil.
\end{equation} 
\begin{lem} If $(\bs,r) \notin \ZZ$, then (\ref{floorceiling}) cannot hold.
\end{lem} 
\begin{proof}  If $(\bb,r) \in \ZZ$, then $(\bs,r)=(\bb,r)$ since $\bs$ is in the closure of $\bb$ and they are connected. Thus $(\bb,r) \in \ZZ \implies (\bs,r) \in \ZZ$.  Let us pick points $s \in \bs$ and $b \in \bb$  respectively. So if $(s,r) \notin \ZZ$ then $(b,r) \notin \ZZ$ and thus $\lfloor (b,r) \rfloor < (b,r)$ while $(s,r) < \lceil (s,r) \rceil$. Thus for (\ref{floorceiling}) to hold, we must have $(s,r) < (b,r)$ and there is an integer namely $\lfloor (b,r) \rfloor = \lceil (s,r) \rceil$ in between. By the intermediate value theorem, the straight line segment $L$ joining $s$ and $b$ must attain this integral value at some point $p \in L$. Now $p \neq s$ (or $b$) because $(s,r) \notin \ZZ$ (resp $(b,r) \notin \ZZ$). We claim that $p \in \bb$. This follows by observing 
\begin{enumerate}
\item  the zero-dimensional facets bordering $\bs$ belong to the set of zero dimensional facets bordering $\bb$.
\item Thus the barycentric coordinates of $s$ and $b$ are strictly positive. 
\item  $p$ is a strict convex combination of $s$ and $b$.   
\end{enumerate}
By these observations, it follows that $p$ admits striclty positive barycentric coordinates on the set of zero-dimensional facets bordering $\bb$. This means that $p \in \bb$. This contradicts $(\bb,r) \notin \ZZ$. 
\end{proof}
By the above lemma therefore any $r$ satisfying (\ref{floorceiling}) must be the vector part of some affine root $\alpha$ belonging to $Y_{\bs}$. For $\alpha \in Y_{\bs}$, if $\alpha(\bb) < 0$ then (\ref{floorceiling}) cannot hold. Summarizing the above observations, the roots generating $\cG_{\bb} (A \otimes R)/\cG^u_{\bs}(A \otimes R)$ correspond to the set $G_{\bs,\bb}=\{ \alpha \in Y_{\bs} | 0 \leq \alpha(\bb)\}$. Now onwards we view $Y_{\bs}$ as a root system. For $\alpha \in G_{\bb,\bs}$, since $\alpha(\bs)=0$, so if $\alpha(\bb)=\alpha(\bs)$, then $-\alpha$ belongs to $G_{\bb,\bs}$ too. Thus the roots corresponding to $G_{\bb,\bs}$ generate a parabolic subgroup of $\cG_{\bs}/\cG^u_{\bs}$ corresponding to $\alpha \in Y_{\bs}$ where the inequality is strict.
\end{proof}



\subsection{Loop groups and their flag varieties} \label{lgpflv}
Let $k$ be a field. Let $\cG$ be a flat affine group scheme of finite type over $k[[t]]$.
For a $k$-algebra $R$ let $R[[t]]$ denote the formal power series ring and $R((t))=R[[t]][t^{-1}]$ the field of Laurent polynomials with coefficients in $R$. Recall that the loop group $L\cG$ of $\cG$ represents the functor mapping $R$ to $\cG(R((t)))$. It is represented by an ind-affine scheme. Recall that the jet group $L^+\cG$ represents the functor mapping $R$ to $\cG(R[[t]])$. It is represented by a closed subscheme of $L \cG$ which is affine. The quotient of fpqc-sheaves $\cF l=L \cG/L^+ \cG$ is the associated flag variety. For a facet $\sigma$,
\begin{equation} \label{flagv}
\cF l_\sigma= L\cG_{\sigma}/L^+\cG_{\sigma}
\end{equation}
is the flag variety associated to $\sigma$. It
is represented by an ind-scheme that is ind-projective over $k$.  It represents the functor that to $R$ associates the set of pairs $(\cE,\theta)$ where $\cE$ is a $\cG_{\sigma}$-torsors on $Spec R[[t]]$ together with a trivialization given by $\theta$ over $Spec R((t))$.  We recall

\begin{thm} \label{etlocsec}  \cite[Theorem 1.4]{pradv} Let $R$ be a strictly henselian ring over $k$. for any point $\Spec(R) \ra \cF l_{\sigma}$, we have $\Spec(R) \times_{\Spec(k)} L^+\cG_{\sigma} \simeq \Spec(R) \times_{\cF l_{\sigma}} L \cG_{\sigma}$. 
\end{thm}

Since $k$ is algebraically closed for us, so by the above theorem, the natural map $L \cG_{\sigma} \ra \cF l_{\sigma}$ is an \'etale local trivial fibration with fibers $L^+ \cG_{\sigma}$.
 
\section{Fundamental group of  the moduli stack}
\subsection{The Bruhat-Tits group scheme} \label{gpsch} Let $\cR \subset X$ be a non-empty finite set of closed points. For each $x \in \cR$, we choose a facet $\sigma_x \subset \cA_T$. Let $\cG_{\sigma_x} \ra \Spec(\hat{\cO_x})$ be the parahoric group scheme corresponding to $\sigma_x$. Let $X^\circ = X \setminus \cR$. For $x \in X$, let $\mathbb{D}_x=\Spec(\hat{\cO_x})$, let $K_x$ be the quotient field of $\hat{\cO_x}$ and let $\mathbb{D}^\circ_x=\Spec(K_x)$. {\it In this paper,  by a Bruhat-Tits group scheme $\cG \ra X$ we shall mean that  $\cG$ restricted to ${X^\circ}$ is isomorphic to $ X^\circ \times G$, and for any closed point $x \in X$,  $\cG$ restricted to $\mathbb{D}_x$ is a parahoric group scheme $\cG_{\sigma_x}$ such that the gluing functions take values in $Mor(\mathbb{D}^\circ_x,G)=G(K_x)$.} This is also the setup of \cite[Defn 5.2.1]{bs}. Let us remark the setup of  \cite{heinloth} and \cite{zhu} are the same. They consider more general group schemes which may not be split over the function field of $X$. The global conditions over $\cG$ demanded in \cite{heinloth} are satisfied in our case (cf. \cite[Introduction]{heinloth}). By \cite[Lemma 5]{heinloth} it is always possible to glue $X^\circ \times G$ with $\{ \cG_x | x \in \cR\}$ along the fpqc cover $\{X^\circ \} \cup \{ \Spec(\hat{\cO_x}) | x \in \cR \}$ of $X$ to get a group scheme $\cG \ra X$. 

The group scheme $\cG \ra X$ depends on the gluing data. For instance, if $E \ra X$ is a principal $G$-bundle, then the adjoint group scheme $Ad(E) \ra X$ is such a group scheme. Its restriction to $X^\circ$ and $\Spec(\hat{\cO_x})$ is always the trivial group scheme, while it may or may not be trivial over $X$. Similarly, let $x_0 \in X$ be a closed point, and let $\cM$ denote the stack of vector bundles of rank $r$ on $X$ and determinant $\cO_X(-d x_0)$ where $0 \leq d <r$. Then for any $V \in \cM$, the adjoint group scheme $Aut(V) \ra X$ is obtained by gluing $X^\circ \times G$ with $\cG_\sigma$ where $\sigma$ is the unique vertex of the alcove $\mathbf{a}$ where only the affine simple root $\alpha_d$ does not vanish.

\subsection{Parahoric torsors} \label{pt}
Let $\cG \ra X$ be a group scheme as in \S \ref{gpsch}. A {\it quasi-parahoric} torsor $\cE$ is a $\cG$--torsor on $X$.  This means that $\cE \times_X \cE \simeq \cE \times_X \cG$ and there is an action map $a: \cE \times_X \cG \ra \cE$ which satisfy the usual axioms for principal $G$-bundles. A {\it parahoric torsor} is a pair $(\cE\, , {\boldsymbol\theta})$ consisting of the pair of a quasi-parahoric torsor and  weights ${\boldsymbol\theta}\,=\, \{\theta_x| x \in \cR \} \in (Y(T) \otimes \RR)^m$ such that $\theta_x$ lies in the facet $\sigma_x$ (cf \S \ref{gpsch}) and $m=|\cR|$.

\subsection{Stack}
 
 Let $\cM_X(\cG)$ denote the moduli stack of $\cG$-torsors on $X$ whose $W$-points, for $W$ a $k$-scheme,  corresponds to the groupoid of $\cG$-torsors $\cE \ra X \times W$. By \cite{heinloth} it is a smooth algebraic stack. It is also irreducible and connected.  
  
We now recall the uniformization theorem for $\cG$-torsors \cite{heinloth}. Set $X^\circ:=X \setminus \cR$. Consider the presheaf of sets on the category of $k$-algebras which to a $k$-algebra $R$ associates
$$Mor(Spec(R) \times_k X^\circ,G).$$ 
Let $L_{X^\circ}(G)$ denote associated sheaf of sets. It is represented by an ind-scheme (cf \cite[Lemma 20]{heinloth}). Set $\cQ= \prod_{x \in \cR} \cF l_{\sigma_x}$. A $R$-point of $\cQ_G$ classifies $\cG$-torsors on $X \times \Spec(R)$ together with a section on $X^\circ \times \Spec(R)$. The map $\cQ \ra \cM_X(\cG)$ forgets the section and the ind-scheme $L_{X^\circ}(G)$ acts on $\cQ$ by changing the section.  By the Uniformization theorem we have an isomorphism of stacks
\begin{eqnarray} \label{unif}
\cQ_G/L_{X^\circ}(G) = \cM_X(\cG) \\
\cQ \times L_{X^\circ}(G) \ra \cQ \times_{\cM_X(\cG)} \cQ.
\end{eqnarray}

\subsubsection{ Construction of $\cG^{\mathbf{a}} \ra \cG$} For each facet $\sigma_x$, let us choose an alcove $\mathbf{a}^x$
such that $\sigma_x$ lies in its closure. Recall that we have assumed that the gluing functions $f_x$ lie in $Mor(\mathbb{D}_x^\circ,G)$. Using $\{f_x\}_{x \in \cR}$ we glue $X^\circ \times G$ with $\cG_{\mathbf{a^x}} \ra \mathbb{D}_x$  to get a global group scheme $\cG^{\mathbf{a}} \ra X$. So the natural maps $\cG_{\mathbf{a^x}} \hra \cG_\sigma$ over $\mathbb{D}_x$ for $x \in \cR$ extends to a morphism $\cG^{\mathbf{a}} \hra \cG$ of Bruhat-Tits group schemes over $X$. In particular, we have a morphism of stacks $\cM_X(\cG^{\mathbf{a}}) \ra \cM_X(\cG)$. The group scheme $\cG^{\mathbf{a}}$ depends on choices of global functions $\{f_x \}_{x \in \cR}$ and alcoves $\mathbf{a}^x$, but in this paper we only need its existence.

\subsubsection{Construction of $\cM_X(\cG^{\mathbf{a}}) \ra \cM_X(G)$} Let us construct a morphism of algebraic stacks $\cM_X(\cG^{\mathbf{a}}) \ra \cM_X(G)$. For each $x \in \cR$, let $w_x$ be an element in the affine Weyl group $W_a$ which maps $\mathbf{a}_0$ (cf \S \ref{alcove}) to $\mathbf{a}^x$ . Set $v_x = w_x v_0$. Let $N$ denote the normalizer of $T$. We choose an element $n_x \in N(K)$ which maps to $w_x$. We will view $n_x$ as an element of $G(K)$.

 Let $E \ra X$ be a principal $G$-bundle obtained by gluing the trivial bundles on $X^\circ$ and $\{\mathbb{D}_x\}_{x \in \cR}$ by $\{n_x f_x\}_{x \in \cR}$. Let $Ad(E) \ra X$ denote the adjoint group scheme of $E$. Then $Ad(E)$ is obtained by gluing the constant group scheme $X^\circ \times G$ with $\cG_{v_0}$ by $\{n_x f_x\}$. Thus $Ad(E)$ is obtained by gluing $X^\circ \times G$ with $\cG_{v_x}$ via $\{f_x\}$.
Since the group scheme $\cG^\mathbf{a}$ is obtained by gluing $X^\circ \times G$ and $\mathcal{G}_{\mathbf{a}}$ via $\{f_x\}$, and we have natural morphisms $\cG_{\mathbf{a}^x} \hra \cG_{v_x}$ so we obtain a natural map of group schemes $\cG^{\mathbf{a}} \ra \Aut(E)$. This also furnishes \begin{equation} \label{phif} \phi_f: \cM_X(\cG^{\mathbf{a}}) \ra \cM_X(Ad(E)).
\end{equation} 
 The principal bundle $E$ is a left $Ad(E)$-torsor and right $G$-torsor on $X$. We have an isomorphism of stacks 
\begin{equation} \label{mapofstacksbasepoint} \mu_E: \cM_X(Ad(E)) \ra \cM_X(G)
\end{equation} which sends a right $Ad(E)$-torsor $\cF$ to the principal $G$-bundle $\cF \times_{Ad(E)} E$. Here $\cF \times_{Ad(E)} E$ denotes the space where   for local sections $f, g, e$ of $\cF$, $Ad(E)$ and $E$ respectively we identify $(e g, g^{-1} f)$ with $(e,f)$. Its inverse is given by sending $F$ to $F \times_G E^{op}$. Here $E^{op}$ has the same underlying space as $E$ but for local section $e$ of $E^{op}$, $e.g$ is defined to be $e g^{-1}$ after viewing it as a local section of $E$. Thus we obtain a morphism of stacks:
\begin{equation} \label{mapfromiwahoritoG}
\mu_E \circ \phi_f: \cM_X(\cG^{\mathbf{a}}) \ra \cM_X(Ad(E)) \ra \cM_X(G).
\end{equation}

\begin{prop}\label{isafibration} The morphism $\pi: \cM_X(\cG^{\mathbf{a}}) \ra \cM_X(\cG)$ satisfies Axiom A (cf \S \ref{fibration}). Hence it is a fibration. For a facet $\sigma$  of the alcove $\mathbf{a}$, let $G^\sigma$ denote the reductive quotient of $\cG_\sigma \otimes k$ and $F^\sigma$ its full flag variety. The fibers of $\pi$ are isomorphic to $\prod_{x \in \cR} F^{\sigma_x}$.
\end{prop}
\begin{proof}  Let $\cF l_{\sigma}$ denote the affine flag variety of the facet $\sigma$ (cf \ref{flagv}). By the uniformization theorem \cite{heinloth}, the following diagram is cartesian:
\begin{equation} \label{cartesiandiagram}
\xymatrix{
\prod_{x \in \cR} \cF l_\mathbf{a} \ar[r]_p \ar[d]^{q_\mathbf{a}} & \prod_{x \in \cR} \cF l_{\sigma_x} \ar[d]^{q_{}} \\
\cM_X(\cG^\mathbf{a}) \ar[r]^\pi & \cM_X(\cG)
}
\end{equation} 
 We first consider the morphism $p$.  By Corollary \ref{borelpar}, the fibers of $p$ are isomorphic to $\prod_{x \in \cR} F^{\sigma_x}$. Since our base field $k$ is algebraically closed, so by Theorem \ref{etlocsec}, the natural quotient morphism $LG \ra \cF l_{\sigma}$ is \'etale locally trivial. So $p$ is also \'etale locally trivial. 

Now we consider the morphism $\pi$ and we will deduce the result for $\pi$ through $p$. By \cite[Thm 4]{heinloth}, for any $S$-family $\cP \in \cM_X(\cG)(S)$, there exists an \'etale covering $S' \ra S$ such that $\cP|_{X^\circ \times S'}$ is trivial.  So the morphism $q$ is \'etale locally trivial. 
By the cartesian diagram (\ref{cartesiandiagram}) it follows that $\pi$ is an \'etale locally trivial morphism of fiber type $\prod_{x \in \cR} F^{\sigma_x}$. Since $\pi$ base-changes to $p$, so the morphism $\pi$ is representable. These stacks are algebraic \cite{heinloth} and locally noetherian. So since $q$ is \'etale locally trivial, so $\pi$ is also a fibration in the sense of \cite{noohi} (cf \S \ref{noohires}). More precisely, $\pi$ is proper and flat and after base-change to any algebraic space the morphism $\pi$ has geometrically connected fibers isomorphic to $\prod_{x \in \cR} F^{\sigma_x}$. So it satisfies Axiom A. Hence it is a fibration in the sense of \cite{noohi}.
\end{proof}

\begin{prop} \label{pi1stack} Let $x: \Spec(k) \ra \cM_X(\cG)$ be any point. The \'etale fundamental group of $\cM_X(\cG)$ is isomorphic to that of $\cM_X(G)$. 
\end{prop}
\begin{proof}   Consider the map (\ref{mapfromiwahoritoG}). By Proposition \ref{isafibration} it is a fibration. By Corollary \ref{borelpar}, its fibers are isomorphic to $\prod_{x \in \cR} G/B$. We choose any $y: \Spec(k) \ra \cM_X(\cG^{\mathbf{a}})$ with image $x$. So by (\ref{fibseq}), it follows that $\cM_X(\cG^{\mathbf{a}}) \ra \cM_X(G)$ induces an isomorphism on $\pi_1$.

Now consider the map $\cM_X(\cG^{\mathbf{a}}) \ra \cM_X(\cG)$. By Proposition \ref{isafibration} it is a fibration. By Corollary \ref{borelpar}, its fibers are isomorphic to $\prod_{x \in \cR} F^{\sigma_x}$. This is connected and simply-connected. We choose any $y: \Spec(k) \ra \cM_X(\cG^{\mathbf{a}})$ with image $x$.  So by (\ref{fibseq}), it follows that this map also induces an isomorphism on $\pi_1$.

\end{proof} 

\begin{Cor} \label{cstack} When $k=\CC$, the fundamental group $\pi_1(\cM_X(\cG),x)$ is trivial.
\end{Cor}
\begin{proof} When $k=\CC$, taking $\cG$ to be the constant group scheme $X \times G$ over $X$, by  \cite[Thm 3.4]{bmp} the \'etale fundamental group of $\cM_X(G)$ is trivial. Then our result follows from Proposition \ref{pi1stack}.
\end{proof}

\section{The moduli space $M_X(\cG)$ of $\cG$ torsors} \label{modulispace}
In \cite{bs} over $\CC$, the notions of (semi)-stability and polystability of parahoric torsors has been defined. The precise definitions are fairly technical and we will not need them.

\subsection{$Z_G$ gerbe} Let us recall the notion of a gerbe banded by a group scheme. Our reference is \cite[\S 2.2]{lieblich}.  To this end, let us recall some notions associated to gerbes. Let $\cX \ra X$ be an algebraic stack over an algebraic space $X$. Recall that the {\it inertia stack} $\cI(\cX) \ra \cX$ is a representable group functor defined by the following condition: Let $a:T \ra \cX$ be a $\cX$-scheme. We have
\begin{equation} \label{inertiastack}
\cI(\cX)(T)=Aut(a:T \ra \cX).
\end{equation}
Explicitly we may construct $\cI(\cX)$ as $\cX \times_{\Delta,\cX \times \cX,\Delta} \times \cX$.

We have a natural functor from the category of sheaves on $X$ to the category of stacks on $X$: given a sheaf $\cF$, for any $U \ra X$ we define the fiber category $\cF_U$ as the discrete category $\cF(U)$. The sheafification $Sh(\cX)$ of $\cX$ is the universal object amongst sheaves $\cF$ on the algebraic space $X$ admitting a map $\cX \ra \cF$ of stacks.

A gerbe on an algebraic space $X$ is a stack $\cX \ra X$ such that $Sh(\cX) \ra X$ is an isomorphism. Equivalently we demand of $\cX \ra X$ that
\begin{enumerate}
\item for every open set $U \ra X$, there exists a covering $V \ra U$ such that the fiber category $\cX_V$ is non-empty, and
\item given an open set $U \ra X$ and any two objects $x,y \in \cX_U$, there exists a covering $V \ra U$ and an isomorphism between $x_V$ and $y_V$.
\end{enumerate}

Let $A$ be an abelian sheaf on the algebraic space $X$. An $A$-gerbe on $X$, or a gerbe banded by a group scheme $A$, is a gerbe $\cX$ alongwith an isomorphism
\begin{equation} \label{gerbebandedbyagroup}
A_{\cX} \simeq \cI(\cX).
\end{equation}
A $1$-morphism $f: \cX \ra \cY$ of stacks between $A$-gerbes is called an isomorphism if the natural morphism $A_{\cX} \simeq \cI(\cX) \ra f^* \cI(\cY) \simeq A_{\cY}$ is the identity.
\begin{prop} \label{isagerbe} The morphism $\cM_X^{rs}(\cG) \ra  M_X^{rs}(\cG)$ is a gerbe banded by $Z_G$.
\end{prop}
\begin{proof}
Consider $\cI(\cM^{rs}) \ra \cM^{rs}$ (cf (\ref{inertiastack})). By definition the fiber over any closed point is isomorphic to $Z_G$. Since $\cI(\cM^{rs})$ has the identity section and $\cM^{rs}$ is connected, so we must have $\cI(\cM^{rs})=\cM^{rs} \times Z_G$. Therefore by definition (cf (\ref{gerbebandedbyagroup})), it follows that $\cM^{rs} \ra M^{rs}$ is a $Z_G$-gerbe.
\end{proof}

\subsection{$\Gamma$-$G$ bundle theory} \label{prelimbs}
In \cite[\S 7]{bbp}, given arbitary {\it real} weights ${\boldsymbol \theta}= \{\theta_x  \in \cA_T\}_{\x \in \cR}$, it is shown how we can find {\it rational} weights ${\boldsymbol \theta'}$ so that a given torsor is (semi)-stable with respect to ${\boldsymbol \theta}$ if and only if it is (semi)-stable with respect to ${\boldsymbol \theta'}$. Varying weights in this sense, we may suppose from now that all weights are rational.  Let $M_X(\cG)$ denote the moduli space of $S$-equivalence classes of semi-stable parahoric $\cG$-torsors.
 
Let $Y \ra X$ be a ramified Galois cover of smooth projective curves with Galois group $\Gamma$. By a $\Gamma$-$G$ bundle $F$ we mean a principal $G$-bundle $F$ on $Y$ together with a lift of $\Gamma$-action. For $y \in \cR$ let $\Gamma_y$ denote the isotropy group at $y$.  Let $\tau_y$ denote the representation 
\begin{equation} \label{typetau} \tau_y : \Gamma_y  \ra \Aut(F_y).
\end{equation} 
Choosing a trivialization of $F_y$, the conjugacy class $\ol{\tau_y}$ of $\tau_y$ with values in $G$ is well determined.
Then $\tau = \{ \ol{\tau_y} | y \in \cR_Y \}$ is called the {\it topological type of $F$}.
The following subsection reduces questions for parahoric torsors with rational weights ${\boldsymbol \theta}$ to $\Gamma$-$G$ bundles of a certain type $\tau$. We refer the reader to \cite{bs} for the correspondence between ${\boldsymbol \theta}$ and $\tau$.
We will not need it.

Let $X$ be a smooth projective curve of genus $g_X \geq 2$. Let $\cG \ra X$ be a Bruhat-Tits group scheme in the sense of (\S \ref{gpsch}). By the main theorem of \cite{bs} there exists a (possibly ramified) finite Galois cover $p: Y \ra X$ with Galois group $\Gamma$ branched at $\cR$ and a principal $G$-bundle $E$ on $Y$ together with a lift of $\Gamma$ action satisfying the following properties:
\begin{enumerate}
\item Let  $y \,\in\, Y$ be a closed point with $x\,:=\,p(y)$. Let $N_y\,=\,{\rm Spec}(B)$, where $B\,=\,
\hat{\cO}_{Y,y}$, and $D_x\,=\,{\rm Spec}(A)$ with $A\,=\,\hat{\cO}_{X,x}$. Let
$U_y$ denote the local unit group i.e the group of local $\Gamma_y$--$G$ automorphisms of $E|_{N_y}$ (cf.
\cite[Definition 2.2.7]{bs}). Then by \cite[Proposition 5.1.2]{bs} we have
\begin{equation}\label{lug}
\cG|_{D_x}(A)\,=\,U_y\, .
\end{equation}
\item Let $Ad(E) \ra Y$ denote the adjoint group scheme of $E$. We have 
\begin{equation} \label{gpscheqn} p_*^{\Gamma} Ad(E) = {\mathcal 
G}.
\end{equation} 
\item    Then we have an isomorphism of moduli stack
\begin{equation} \label{sbs}
\cM_Y^{\tau}(\Gamma, G) = \cM_X(\cG)
\end{equation}
given by sending a $\Gamma$-$G$ bundle $F$ to $p_*^\Gamma (F \times_G E^{op})$. In the reverse direction one sends $\cF \in \cM_X(\cG)$ to $p^*(\cF) \times_{p^* \cG} E$. Observe here that $p^* \cG$ acts on $E$ on the left by adjunction and that the topological type of $p^*(\cF) \times_{p^* \cG} E$ is that of $E$. 
\item under the above isomorphism (semi)-stable  $\Gamma$-$G$ bundles are identified with (semi)-stable $\cG$-torsors respectively. Furthermore the above mentioned identification passes to $S$-equivalence classes and we get
\begin{equation} \label{msi}
M_Y^{\tau}(\Gamma,G) = M_X(\cG).
\end{equation}

\end{enumerate}

Let $m$ be the number of branch points of $p$ and let $\{n_i\}$ be the corresponding ramification indices. Let $K_G$ denote the maximal compact subgroup of $G$ and $\pi$ be a Fuchsian group generated by $2 g_X +m$ elements as follows:
\begin{equation} \label{generators} \pi= \frac{<A_1,B_1,\cdots,A_g,B_g, C_1,\cdots,C_m>}{<A_1B_1A_1^{-1}B_1^{-1} \cdots A_g B_g A_g^{-1} B_g^{-1} C_1 \cdots C_m=1,C_i^{n_i}=1>}
\end{equation} 

Let $K_G$ denote the maximal compact subgroup of $G$. By \cite[Corollary 8.1.8(1)]{bs}, a stable $\Gamma$-$G$ bundle $E$ corresponds to an irreducible unitary representation $\rho: \pi \ra K_G$.  
{\it We shall call the above set of equivalences $\Gamma$-$G$ bundle theory}. 
\subsection{Applications of $\Gamma$-$G$ bundle theory} \label{appli}

We recall that a principal $G$ bundle  $F$ is called {\it regularly stable} if $F$ is stable and the natural morphism $Z_G \ra \Aut(F)$ is an isomorphism. We shall say that a $\cG$ torsor $\cF$ is regularly stable if under the isomorphism \ref{sbs} the corresponding $\Gamma$-$G$ bundle $F$ on $Y$ is regularly stable. When the weights are {\it real} this definition is independent of the nearby rational weight chosen as above. This follows immediately from \cite[Prop 7.5]{bbp}. 




 \begin{prop}  \label{codimestimate} Let $G$ be a semi-simple simply-connected group. Over $\CC$, when $g_X \geq 3$,
the open substack $\cM_X^{rs}(\cG)$ of regularly stable torsors has complement of co-dimension at least two.
\end{prop}

\begin{proof} We divide the proof  in two steps. In step one, we show that the locus of stable but not regularly stable torsors has codimension at least two. In step two, we show that the non-stable torsors have codimension at least two.

Let $\cE$ be stable but not regularly stable torsor. Since we can assume that our weights $\{\theta_x | x \in \cR \}$ are {\it rational}, so  let $\cE$ corresponds to $E \ra Y$ on $p: Y \ra X$ by $\Gamma$-$G$ bundle theory.  Let  $Aut(E)$ (resp. $Aut_Y(E)$) denote the group of automorphisms of $E$ as a $\Gamma$-$G$ (resp. principal $G$-) bundle. Now $Aut_Y(E)$ is an algebraic group (cf \cite[page 227]{bbn2005}) and therefore $Aut(E)=Aut_Y(E)^\Gamma$ is also an algebraic group.  By a straightforward generalization of \cite[Prop 3.2]{ramanathan75}, it follows that $Aut(E)$ is a finite group since $G$ is semi-simple and $E$ is $\Gamma$-$G$ stable. Let us describe a map from $Aut(E)$ to the set of conjugacy classes of finite order elements in $G$.

Let $Ad(E)$ denote the adjoint group scheme of $E$. Let us temporarily choose a point $y \in Y$ and an isomorphism $\theta: Ad(E)_y \simeq G$.  Let 
\begin{equation}
e_y: Aut(E) \ra Ad(E)_y \stackrel{\theta}{\lra} G
\end{equation}
denote the evaluation map. It is a morphism of algebraic groups (cf \cite[Proof of Prop 2.4]{bbn2005}). Further under $e_y$ the natural inclusion $Z_G \hra Aut(E)$ surjects onto $Z_G \subset G$. Let $\psi \in Aut(E)$ be an arbitrary element and  let 
\begin{equation} 
e_y(\psi)=g. \end{equation} So the conjugacy class of $g$ is independent of the choice of the identification $\theta$ of $Ad(E)_y$ with $G$. Further, $g$ is a semi-simple element because it has finite order. We fix  a maximal torus $T$ in $G$ containing $g$. Let $W(T)$ denote the corresponding Weyl group. The space of conjugacy classes of semi-simple elements is parametrized by the quotient $T/W(T)$.  Since $Y$ is projective and $T/W(T)$ is affine, so the conjugacy class of $g$ is independent of $y$ too. Consider the morphism $\theta: E \ra G$ defined by the relation $\psi(e)=e \theta(e)$ for $e \in E$. Notice that for any $e \in E$, $\theta(e)$ belongs to the conjugacy class of $g$.

Let $Z_g$ denote the centralizer of $g$ in $G$. Let us describe the group $Z_g$ and show that $E$ admits a $\Gamma$-equivariant reduction of structure group to $Z_g$. Since the case $\psi \in Z_G$ is trivial, so let us assume that $\psi \in Aut(E) \setminus Z_G$.

By \cite[Prop 0.35]{dignemichel}, the centre $Z_G$ is the intersection in $T$ of the kernels of the roots relative to $T$. Therefore if $\psi \in Aut(E) \setminus Z_G$, then there is a root relative to $T$ not vanishing on $g$.  Let $Z_{g0}$ denote the connected component of identity in $Z_g$.
 By \cite[2.7, 2.8]{steinberg65}, the centralizer $Z_{g0}$ is a connected reductive group of the same rank as $G$ which has finite index in $Z_g$. Further, it is generated by $T$ and $X_\alpha$ for roots $\alpha$ (relative to $T$) such that $\alpha(g)=1$. Therefore $Z_g$ is a proper subgroup of $G$ of strictly smaller dimension when $\psi \in \Aut(E) \setminus Z_G$.

We now adapt the proof of \cite[Prop 2.4]{bbn2005} to show the reduction of structure group.  Let $S \subset E$ denote the subvariety of $E$ defined by $e \in E$ such that $\theta(e)=g$. One can check that for any $y \in Y$ the fiber $S_y$ is non-empty and that $Z_g$ acts transitively on $S$. So $S$ defines a reduction of structure group of $E$ from $G$ to $Z_g$. Further since $\psi \in Aut(E)=Aut_Y(E)^\Gamma$, so this reduction is $\Gamma$-equivariant. In conclusion, $E$ comes from an extension of structure group of a $\Gamma-Z_g$-bundle $S$. Since $E$ comes from an irreducible unitary representation $\rho: \pi \ra K_G$, therefore $S$ also comes from an irreducible unitary representation $\rho_g: \pi \ra K_{G_{g}}$. In particular, $\rho$ factors via  $\rho_g$ and $K_{Z_g} \hra K_G$.
Notice that the local representation type of $S$ gets fixed by that of $E$.

We now prooceed to estimate the dimensions. Since $E$ is stable but not regularly stable, so we will assume now that $\psi \in Aut(E) \setminus Z_G$. In particular, $g \in G \setminus Z_G$. For each ramification point $y_i \in Y$, let $\rho_{y_i}: \Gamma_{y_i} \ra \Aut(E_{y_i})$ be the local isotropy representation (cf (\ref{typetau})). Choosing an isomorphism $\Aut(E_{y_i}) \simeq G$, since $\Gamma_{y_i}$ is a finite order element, so we may consider $\rho_{y_i}: \Gamma_{y_i} \ra T$, if we  need after conjugation.  Recall that $C_i$ (cf (\ref{generators})) denote finite order generators of the Fuchsian group $\pi$. We identify $C_i$ with a generator of $\Gamma_{y_i}$. Let us denote
\begin{equation} e_G(C_i)= rank(Id-\rho_{y_i}(C_i))
\end{equation} on $Lie(K_G)$. When $G$ is semi-simple (not necessarily connected or simply-connected), by \cite[Cor 8.1.12 and Prop 7.1.1]{bs}) the dimension of the moduli space of parahoric torsors or equivalently $\Gamma$-$G$ bundles of fixed topological type $\tau$ is equal to 
\begin{equation} \label{dimformula}
dim_{\CC}(G) (g_X-1) +\frac{1}{2} \sum_{i=1}^m e_G(C_i)
\end{equation}
This formula is independent of the choice of isomorphisms $\Aut(E_{y_i}) \simeq G$. Since a generic $\Gamma$-$G$ bundle is regularly stable, so by Proposition \ref{isagerbe} the dimension of the moduli stack of $\Gamma$-$G$ bundles is the same as that of the moduli space which is given by the formula (\ref{dimformula}). We now want to deduce a formula for $\Gamma$-$Z_g$ bundles because $Z_g$ is only reductive. To this end, set $Z_{g}^s=Z_g/centre(Z_g)$ and $Z_g^a=Z_g/[Z_g,Z_g]$ as the semi-simple and abelian quotients of $Z_g$ respectively. The natural projection  
  \begin{equation} Z_g \ra Z_g^s \times Z_g^a
  \end{equation} has finite kernel and cokernel. So the dimension of the stack of $\Gamma$-$Z_g$ bundles of a fixed topological type equals the sum of the dimensions of the stacks of $\Gamma$-$Z_g^s$ and $\Gamma$-$Z_g^a$ bundles of the corresponding topological type. We further have \begin{equation} Lie(K(Z_g))=Lie(K(Z_g^s)) \oplus Lie(K(Z_g^a)).
  \end{equation}  We have $$ e_G(C_i)=rk(Id_{Lie(K_{G})}- \rho(C_i)) \geq rk(Id_{Lie(K_{Z_g})}- \rho_g(C_i))$$ because $Id_{Lie(K_{G})}-\rho(C_i)$ on $Lie(K(Z_g))$ restricts to $Id_{Lie(K_{Z_g})}- \rho_g(C_i)$.  Let $\rho^s: \pi \ra K(Z_g^s)$ be the composite $ \pi \stackrel{\rho_g}{\lra} K(Z_g) \ra K(Z_g^s)$. Since 
$\rho_g(C_i)$ on $Lie(K(Z^s_g))$ restricts to $\rho_g^s(C_i)$ but on $Lie(K(Z^a_g))$ restricts to zero, so this  last term is at least $$ rk(Id_{Lie(K(Z^s_g))}- \rho^s(C_i))+ dim_{\mathbb{R}}(Lie(K(Z_g^a))),$$ which equals $e_{Z_g^s}(C_i)+ dim_{\mathbb{C}}(Z_g^a)$.  Therefore we have
\begin{equation}
e_G(C_i) \geq e_{Z_g^s}(C_i)+ dim_{\mathbb{C}}(Z_g^a).
\end{equation}

On the other hand, since $Z_g^a$ is a torus, so $\Gamma$-$Z_g^a$ bundles of a fixed topological type, correspond to a direct sum of $\Gamma$-line bundles of a fixed topological type. These correspond to a direct sum of parabolic line bundles of fixed parabolic type. Hence their dimension is equal to $dim(Z_g^a)(g_X-1)$.

Thus from the dimension formula (\ref{dimformula}) we have the dimension of $\cM^{rs}_X(\cG)$ minus the dimension of the substack of stable bundles admitting reduction of structure group to $Z_g$ equals
 \begin{eqnarray*} dim(G) (g_X-1) +\frac{1}{2} \sum_{i=1}^m e_G(C_i)-[dim(Z_g^s) (g_X-1)  +\frac{1}{2} \sum_{i=1}^m e_{Z_g^s}(C_i)] - dim(Z_g^a)(g_X-1) \\ \geq \{dim(G)- dim(Z_g^s) -dim(Z_g^a)\} (g_X-1) + \frac{m}{2} dim(Z^a_g). 
 \end{eqnarray*}
  Let $k$ be the number of simple roots of $G$, say relative to $T$, that do not take the value one on $g$. Then $k=dim(Z_g^a)$, $k \geq 1$ and we have $dim(Z_g^s) \leq dim(G) - 3k$. Therefore the above sum is at least $k(2 (g_X-1) +\frac{m}{2})$.

 Since $m \geq 0$, $k \geq 1$  and $g_X \geq 3$ for us, so this is bounded below by four. 
 
Now let us show step two. Let us assume that $\cE$ is not a stable $\cG$-torsor. So $E$ is   not a stable $\Gamma$-$G$ bundle. Hence $E$ admits a $\Gamma$-equivariant reduction $E_P$ of structure group to a parabolic subgroup $P$ of $G$. Since $G$ is semi-simple, so this means that $deg(E_P \times_P \mathfrak{p}) \geq 0$ and the degree of $E (\gfr)$ is zero.  Let $H^0(Y,\Gamma,-)$ denote the space of $\Gamma$-invariant sections of a $\Gamma$-sheaf on $Y$ and let $H^i(Y,\Gamma,-)$ denote its higher derived functors. 

Since  $E_P \times_P \mathfrak{p} \hra E \times_G \mathfrak{g}$ is a sub-bundle, so
\begin{equation} \label{h0estimate} h^0(Y,\Gamma,E_P \times_P \mathfrak{p}) \leq h^0(Y,\Gamma,E(\mathfrak{g})).
\end{equation} Now the tangent space at $E$ is isomorphic to $H^1(Y, \Gamma, E(\gfr))$. Similarly the tangent space at $E_P$ of the stack of $\Gamma$-$P$ bundles is $H^1(Y,\Gamma,E_P \times_P \mathfrak{p})$. Hence by (\ref{h0estimate}) we get
$$
h^1(E(\gfr))-h^1(E_P \times_P \mathfrak{p}) \geq \chi (E(\gfr)) - \chi (E_P \times_P \mathfrak{p}).$$ Let $pdeg$ denote the parabolic degree of the associated parabolic vector bundles. So by Riemann-Roch, this last expression becomes
$$dim(G) (g_X-1) - pdeg(E(\gfr)) - [dim(P) (g_X-1) - pdeg(E_P \times_P \mathfrak{p})]$$ $$ =
[dim(G)-dim(P)](g_X-1) + pdeg(E_P \times_P \mathfrak{p}) - pdeg(E(\gfr))  \geq g_X-1 \geq 2
.$$ The second last inequality follows from  $pdeg(E(\gfr))=0$ and $pdeg(E(\mathfrak{p})) \geq 0$.
\end{proof}

\section{Fundamental group of moduli spaces}
\subsection{Arbitrary characteristic} 
\begin{prop} When $g_X \geq 4$, we have a surjection $\pi_1(\cM_X(G)) \ra \pi_1(M_X(G))$.
\end{prop}
\begin{proof} By \cite[Lemma 2.1 and Theorem 2.5 (ii)]{biswashoffmann2012}, the codimesion of the regularly stable stack is at least two. By Proposition \ref{codim2}, we have $\pi_1(\cM_X(G)) = \pi_1(\cM^{rs}_X(G)$. By Proposition \ref{isagerbe}, $\cM^{rs}_X(G) \ra M_X^{rs}(G)$ is a gerbe banded by $Z_G$. Let $BZ_G=[\Spec(\CC)/Z_G]$ denote the classyfying stack of $Z_G$. Choosing appropriate base points, by $\pi_0(\Spec(k)) \ra \pi_0(BZ_G) \ra *$, we see that $BZ_G$ is connected. Thus $\pi_1(\cM_X^{rs}(G))$ surjects onto $\pi_1(M_X^{rs}(G))$. Since the moduli space $M_X(G)$ is a normal scheme, so $\pi_1$ of $M_X^{rs}(G)$ surjects onto it.
\end{proof}

\subsection{Over $\CC$} We place ourselves now over the complex numbers.
Let $M_X(\cG)$ denote $S$-equivalence classes of semi-stable of semi-stable $\cG$-torsors. 
\begin{prop} The \'etale fundamental group of $M_X(\cG)$ is trivial.
\end{prop}
\begin{proof} By Corollary \ref{cstack}, $\pi_1(\cM_X(\cG))=e$. By Proposition \ref{codimestimate}, when $genus(X) \geq 3$, then the open substack $\cM_X^{rs}(\cG)$ of $\cM_X(\cG)$ consisting of regularly stable torsors has complement of codimension at least two. By Proposition \ref{codim2}, it follows that the \'etale fundamental group of $\cM^{rs}$ is trivial.  By Proposition \ref{isagerbe}, the morphism
\begin{equation}
\cM^{rs}(\cG) \ra M^{rs}_X(\cG)
\end{equation}
is a gerbe banded by $Z_G$. It can be checked directly that this  map satisfies the conditions to be a fibration of algebraic stacks of \S \ref{noohires} i.e after any base-change the sequence (\ref{fibseq}) remains exact.  Therefore choosing the appropriate base points, by the fibration sequence  $ \cdots \ra \pi_1(\cM_X^{rs}(\cG)) \ra \pi_1(M^{rs}_X(\cG)) \ra \pi_0(BZ_G) \ra \cdots$ it follows that the \'etale fundamental group of $M^{rs}_X(\cG)$ is trivial. 

By $\Gamma$-$G$ bundle theory, the moduli space $M_X(\cG)$ is a normal projective variety by \cite[Theorem 8.1.7]{bs}. It admits $M^{rs}_X(\cG)$ as a simply-connected non-empty open subvariety.  For a normal variety, any smooth open subvariety will induce surjection on fundamental group. Let $U$ be the smooth subvariety of $M^{rs}_X(\cG)$. Thus $\pi_1(U)$ surjects onto $\pi_1(M_X(\cG))$. Since this homomorphism factors through $\pi_1(M^{rs}_X(\cG))$, which is trivial, so it follows that $M_X(\cG)$ is simply-connected.
\end{proof}

\bibliographystyle{plain}
\bibliography{fungp}
 \end{document}